\documentclass[11pt]{article}

\usepackage[left=2.2cm, right=2.2cm, top=2.5cm, bottom=2.5cm]{geometry}

\usepackage{amsfonts, amsmath, hyperref, latexsym, amssymb, amsthm, url,tikz, xcolor, bm, ifthen}
\usepackage{setspace}
\usepackage{caption}
\usetikzlibrary{shapes.geometric}
\usetikzlibrary{positioning}
\newtheorem{theorem}{Theorem}[section]
\newtheorem{lemma}[theorem]{Lemma}

\newtheorem{definition}[theorem]{Definition}

\newtheorem{remark}[theorem]{Remark}

\newcommand{\condE}[2]{{\mathbb{E}\left[ \left. #1  \right\rvert #2 \right]}}

\title{Depths in random recursive metric spaces}
\author{Colin Desmarais\thanks{Supported by the Swedish Research Council, the Knut and Alice Wallenberg Foundation, and the Swedish Foundation's starting grant from the Ragnar S\"oderberg Foundation.}\\
{\small Department of Mathematics}\\[-0.8ex]
{\small Uppsala University}\\[-0.8ex] 
{\small Uppsala, Sweden}\\
{\small\tt colindesmarais@gmail.com}}
\begin{document}

\maketitle

\begin{abstract}
As a generalization of random recursive trees and preferential attachment trees, we consider random recursive metric spaces. These spaces are constructed from random blocks, each a metric space equipped with a probability measure, containing a labelled point called a hook, and assigned a weight. Random recursive metric spaces are equipped with a probability measure made up of a weighted sum of the probability measures assigned to its constituent blocks. At each step in the growth of a random recursive metric space, a point called a latch is chosen at random according to the equipped probability measure and a new block is chosen at random and attached to the space by joining together the latch and the hook of the block. We use martingale theory to prove a law of large numbers and a central limit theorem for the insertion depth; the distance from the master hook to the latch chosen. We also apply our results to further generalizations of random trees, hooking networks, and continuous spaces constructed from line segments. 
\end{abstract}
%--------------------------------------------%
%---------- INTRODUTION ----------%
%--------------------------------------------%
\section{Introduction}

A well-studied class of random trees are random recursive trees, whereby a sequence of trees $T_0, T_1, T_2, \ldots$ is grown by attaching a new child to a vertex chosen uniformly at random to construct the successor. In this work, we generalize the model by constructing a sequence $\bm{G}_0, \bm{G}_1, \bm{G}_2, \ldots$ of metric spaces at random. At each step $n$ in the growth process, a point called a {\em latch} is chosen at random according to a probability measure on the space $\bm{G}_{n-1}$. A new randomly chosen metric space $\bm{B}_n$ called a {\em block}, equipped with a probability measure, is attached to the latch via a labelled point in $\bm{B}_n$, thereby creating a new metric space $\bm{G}_n$ with a new probability measure defined as a weighted sum of the probability measures on $\bm{G}_{n-1}$ and $\bm{B}_n$. Since the growth process generalizes the growth process of random recursive trees, we call $\bm{G}_0, \bm{G}_1, \bm{G}_2, \ldots$ a sequence of {\em random recursive metric spaces}. These random recursive metric spaces, defined more precisely below, resemble models of random metric spaces introduced in \cite{SENI:19}, though in contrast to the present work, the blocks in $\cite{SENI:19}$ are renormalized by a constant that vanishes with $n$. If each block consists of two vertices joined by an edge and the attachment of the block is done by fusing one of the vertices of the block with the latch, and the probability measure on the space is uniform on all vertices, then we construct random recursive trees as discussed above. 

Several previously studied generalizations of random recursive trees are encompassed in our model as well. These include for example preferential attachment trees. Similar to random recursive trees, the growth process is described by attaching an edge at each step by fusing one vertex of the edge to a vertex $v$ chosen at random, but in this case the choice of $v$ is made proportional to the outdegree of $v$. Another generalization of random recursive trees is the weighted random recursive tree. In this context, every time a new vertex appears in the tree it is assigned a random weight, and the choice of the new parent at each step is made proportional to the weight of the vertex. If instead of attaching a single edge at each step, an entire new graph is attached, the resulting model is a hooking network. In all of these examples, the {\em insertion depth}, that is, the distance to the root from the vertex chosen at random in each step in the growth process, was shown to follow a normal limit law once scaled by $\sqrt{\ln n}$ (see for example \cite{DEVR:88,DEMA:21, ESLA:22, LODE:22, MAHM:91, MAHM:92, MABR:19, SENI:21}, and many more). The main result of this work is Theorem $\ref{thm:main}$, where we prove a law of large numbers and a central limit theorem for the insertion depth in random recursive metric spaces. 

In Sections \ref{sec:whmps} and \ref{sec:model}, the description of the model introduced above is made more precise. The main result of this work (Theorem \ref{thm:main}) is stated in Section \ref{sec:main} along with a brief outline of the proof method. Examples of our model, including random recursive trees, preferential attachment trees, weighted random recursive trees, hooking networks, and constructions from line segments, and how Theorem \ref{thm:main} applies, are included in Section \ref{sec:examples}. The proof of Theorem \ref{thm:main} is contained in Section \ref{sec:proofs}.

%----------------------------------------------------------------%
%---------- DEFINITION OF THE MODEL ----------%
%----------------------------------------------------------------%
\subsection{Weighted hooked metric probability spaces}\label{sec:whmps}

We define a {\em weighted hooked (complete separable) metric probability space} to be a 5-tuple $(b, d, h, p, w)$ in the following way: $(b,d)$ is a complete and separable metric space and $h$ is a point identified in $b$ called a {\em hook}, $p$ is a Borel probability measure on $b$, and $w$ is a real number called a {\em weight}. We wish to define a probability space of weighted hooked metric probability spaces. We start with the space $\mathcal{X}$ of equivalence classes under isomorphism (bijective isometry) of (complete separable) metric measure spaces with full support (the support of the measure is the whole space). Gromov proved \cite{GROM:99} that $\mathcal{X}$ along with the Gromov-Prohorov metric $d_{GP}$ is a complete separable metric space. We extend the Gromov-Prohorov metric by first defining the hooked (or rooted) Gromov (pseudo)distance $d_{GP}^\bullet$ on the space of hooked metric measure spaces by
\[ d_{GP}^\bullet\left( (b_1, h_1, d_1, p_1), (b_2, h_2, d_2, p_2)\right) := \inf_{E, \phi_1, \phi_2} \max\left((d_{P}\left( \phi_{1,\#}(p_1), \phi_{2,\#}(p_2)\right), d_E\left(\phi_1(h_1), \phi_2(h_2)\right)\right),\]
where the infimum is taken over all complete separable metric spaces $(E, d_E)$ such that there exists isometric embeddings $\phi_1:b_1 \rightarrow E$ and $\phi_2 : b_2 \rightarrow E$ from the complete separable metric spaces $(b_1, d_1)$ and $(b_2, d_2)$ to $(E, d_E)$. The points $h_1$ and $h_2$ are labelled points (hooks) in $b_1$ and $b_2$ respectively. The measures $\phi_{1,\#}(p_1)$ and $\phi_{2,\#}(p_2)$ on $E$ are the push-forward measures of $p_1$ and $p_2$ respectively, while $d_P$ is the Prohorov metric on the measures on $E$. Let $\mathcal{X}^\bullet$ be the space of equivalence classes under hook preserving isomorphisms of hooked metric measure spaces with full support, that is, $(b_1, d_1, h_1, p_1)$ and $(b_2, d_2, h_2, p_2)$ belong to the same equivalence class if there exists a bijective isometry $\phi:b_1 \rightarrow b_2$ such that $\phi(h_1) = h_2$. By using similar arguments as for the Gromov-Hausdorff-Prohorov metric for rooted compact metric measure spaces in \cite{ABDH:13}, one can show that $(\mathcal{X}^\bullet, d_{GP}^\bullet)$ is a complete and separable metric space (the only difference in the argument is we can omit the ``Hausdorff" part). We use the word `hook' in this work, but it serves the same purpose as the `root' in \cite{ABDH:13}. Finally, since $\mathcal{X}^\bullet$ and $\mathbb{R}$ are Polish spaces, so is $\mathcal{X}^\bullet \times \mathbb{R}$, and so we can define a Borel probability measure on this space. We may now define a random {\em block} $\bm{B} = (B, D', H', P', W)$ as a random element of $\mathcal{X}^\bullet \times \mathbb{R}$ (according to some Borel probability measure). 

\subsection{Definition of random recursive metric spaces}\label{sec:model}

Let $\bm{B}$ be a random block as defined in the previous section such that $W \geq 0$ and $\mathbb{P}(W=0) < 1$. For $n \geq 1$, let $\bm{B}_n = (B_n, D_n', H_n', P_n', W_n)$ be independent and identically distributed copies of $\bm{B}$ and let $\bm{B}_0 = (B_0, D_0', H, P'_0, W_0)$ be a random block which may or may not have the same distribution as $\bm{B}$. We add the condition that $W_0 > 0$; notice that $W_0$ might not have the same distribution as $W$.

 Formally, random recursive metric spaces $\bm{G}_n = (G_n, D_n, H, P_n, S_n)$ are random elements  from $\mathcal{X}^\bullet \times \mathbb{R}$ constructed recursively from the blocks $\bm{B}_n$. We start by setting $\bm{G}_0 = (G_0, D_0, H, P_0, S_0) = \bm{B}_0$. The identified point $H$ serves as the master hook of all following random recursive metric spaces. At each step $n \geq 0$, we grow $\bm{G}_{n+1}$ from $\bm{G}_n$ by randomly choosing a point $v_{n+1} \in G_n$ called a {\em latch} according to the probability $P_n$. Next we identify together the latch $v_{n+1}$ with the hook $H'_{n+1}$ to form $G_{n+1}$. We let $D_{n+1}$ be the concatenation of the metrics $D_n$ and $D'_{n+1}$ in the canonical way and set $S_{n+1} = S_n + W_{n+1}$. Define $P_{n+1}$ to be the weighted sum of the probabilities $P_n$ and $P'_{n+1}$; more precisely, extend the probability measures to all of $G_{n+1}$ such that for any event $A \subseteq G_{n+1}$, the extensions $\widehat{P_n}$ and $\widehat{P'_{n+1}}$ satisfy $\widehat{P_n}(A) = P_n(A\cap G_{n})$ and $\widehat{P'_{n+1}}(A) = P'_{n+1}(A\cap B_{n+1})$ and define 
\begin{equation}\label{eq:probabilitysum}
 P_{n+1} = \frac{S_n}{S_{n+1}}\widehat{P_n} + \frac{W_{n+1}}{S_{n+1}}\widehat{P'_{n+1}}.
\end{equation}
Set $\bm{G}_{n+1} = (G_{n+1}, D_{n+1}, H, P_{n+1}, S_{n+1})$. 

By the recursive nature of \eqref{eq:probabilitysum}, the probability measure $P_n$ is a weighted sum of extensions of the probability measures $P_k'$ of the constituent blocks $\bm{B}_k$ that make up $\bm{G}_n$. In a slight abuse of notation, we will write 
\[ P_n = \sum_{k=0}^n \frac{W_k}{S_n} P'_k.\]

Models with constructions similar to random recursive metric spaces have been studied. These include models where the blocks consist of finite line segments \cite{ADGO:17, BOVA:06,CUHA:17, HAAS:17}, and constructions of iterative gluing of metric spaces introduced by S\'enizergues \cite{SENI:19, SENI:20}. In fact, the notation for random recursive metric spaces presented in this paper is heavily inspired by the notation introduced in \cite{SENI:19}. In most of these models, the blocks are compact and scaled by a factor of roughly $n^{-\alpha}$ for some $\alpha >0$ as they are attached. Under certain assumptions, it is shown that the probability measures $P_n$ converge weakly to some limiting measure and that the limiting metric space is compact. Since our blocks are identically distributed, the growth of the random recursive metric space is unbounded (and so do not have a compact limit).

%------------------------------------------%
%---------- MAIN RESULT ----------%
%------------------------------------------%
\subsection{Main Result}\label{sec:main}

We start by defining the insertion depth of a block in a random recursive metric space and the random depth within a block

\begin{definition}
For a sequence of random recursive metric spaces $\bm{G}_0, \bm{G}_1, \bm{G}_2, \ldots,$ define the (random) {\em insertion depth} $\Delta_n$ of the block $\bm{B}_n$ as the distance from the master hook to the $n$'th latch $v_n$. That is, 
\[ \Delta_n = D_n(H,v_n).\]
\end{definition}

\begin{definition}
Given the block $\bm{B}_n$, define the {\em random depth} $\Delta'_n$ within $\bm{B}_n$ to be the distance $D_n'(H'_n, U)$ from the hook $H_n'$ to a point $U \in B_n$ chosen at random according to $P_n'$. 
\end{definition}

Note that for $n \geq 1$, all $\Delta_n'$ are identically and independently distributed to a random variable $\Delta'$. The distance functions $D_n(H,x)$ and $D'_n(H_n, x)$ are continuous functions on their respective metric spaces, and so $\Delta_n$ and $\Delta'_n$ are measurable random variables. We are now ready to state the main result of this work. 
\begin{theorem}\label{thm:main}
Suppose that $\mathbb{E}\left[W^2\right], \mathbb{E}\left[(W\Delta')^2\right]$, and $\mathbb{E}\left[W(\Delta')^2\right]$
are all finite. Then 
\[ \frac{\Delta_n}{\ln n} \xrightarrow{p} \frac{\mathbb{E}\left[W\Delta'\right]}{\mathbb{E}[W]}\]
and 
\[ \frac{\Delta_n - \mathbb{E}[W]^{-1}\mathbb{E}\left[W\Delta'\right]\ln n}{\sqrt{\ln n}} \xrightarrow{d} \mathcal{N}\left(0,\frac{ \mathbb{E}\left[W(\Delta')^2\right]}{\mathbb{E}[W]}\right).\]
\end{theorem}

\begin{remark}
Note that we make no assumptions on $\mathbb{E}\left[(\Delta')^2\right]$ nor on $\mathbb{E}\left[\Delta'\right]$ for Theorem \ref{thm:main}. Of course if $W$ and $\Delta'$ are independent, then moment conditions on $\Delta'$ are required. 
\end{remark}

\begin{remark}
It may be tempting to hope for almost sure convergence instead of convergence in probability in Theorem \ref{thm:main}. But almost sure convergence doesn't hold in even the simplest case, random recursive trees. Devroye proved that for the insertion depth $\Delta_n$ in random recursive trees, $\Delta_n/\ln n \xrightarrow{p} 1$ \cite{DEVR:88}. But it is well known that the degree of the root in random recursive trees is unbounded, so $\Delta_n = 0$ occurs infinitely often. Even more, we know that $H_n = \max\{\Delta_0, \Delta_1, \ldots, \Delta_n\}$ satisfies $H_n/\ln n \xrightarrow{a.s.} e$ \cite{PITT:94}, and so almost sure convergence of the insertion depth cannot hold. 
\end{remark}

To prove Theorem \ref{thm:main}, we make use of an observation present in many works on insertion depths in random recursive trees and their generalizations. For each $n$, define independent Bernoulli random variables $Y_{n,0}, Y_{n,1}, Y_{n,2}, \ldots, Y_{n, n-1}$ such that 
\[ Y_{n,k} \sim \text{Be}\left(\frac{1}{k+1}\right).\]
It was shown in \cite{DEVR:88} that the insertion depth $\Delta_n$ of the $n$'th vertex in a random recursive tree is distributed as the number of records in a uniform random permutation, and so has distribution $\Delta_n \sim \sum_{k=0}^{n-1} Y_{n,k}$. An application of Lindeberg's condition on the sum of the independent random variables $Y_{n,k}$ yields a normal limit law for $\Delta_n$. The observation that the insertion depth can be written as a sum of Bernoulli random variables appears in several works, see for example \cite{CUHA:17, DOBR:96, DRMO:09, ESLA:22, HAAS:17, KUWA:10, LODE:22, MABR:19, SENI:19, SENI:21}. In this work, we show that the insertion depth for random recursive metric spaces is a sum of a product of $\Delta'_{n,k} \sim \Delta'_k$ and Bernoulli random variables $J_{n,k}$, this time with success probability $W_k/S_k$ (see Lemma \ref{lem:buckets}). We then approximate these sums of random variables with martingales that satisfy the conditions necessary to apply a martingale central limit theorem, specifically \cite[Corollary 3.1]{HAHE:80}.

%--------------------------------------%
%---------- EXAMPLES ----------%
%--------------------------------------%
\section{Applications of Theorem \ref{thm:main}}\label{sec:examples}

In this section we outline how random recursive metric spaces generalize models of preferential attachment trees, hooking networks, and constructions of iterative gluing of line segments. We recover previous results on the insertion depths, and further generalize these models in natural ways to prove new results. The applications listed here are by no means exhaustive, but are meant to be illustrative of the generality of random recursive metric spaces and of Theorem \ref{thm:main}. 

%----------- Random Trees -----------%
\subsection{Models of random trees}\label{sec:trees}

Consider the following generalization of random recursive trees. Set $\alpha \geq 0$, let $A, A_1, A_2, A_3, \ldots$ be i.i.d. nonnegative random real numbers such that $\mathbb{P}(A= 0) < 1$, and let $A_0$ be a strictly positive random real number. A sequence of trees $T_0, T_1, T_2, \ldots$ is constructed with an initial tree $T_0$ consisting of a single vertex $v_0$ acting as the root of all trees that follow. At each step $n \geq 0$ in the growth process, given $A_0, A_1, \ldots, A_n$ and $T_n$, a vertex $v_k$ is chosen at random from $T_{n}$ with probability 
\begin{equation}\label{eq:prefatt}
 \frac{\alpha \deg^+(v_k) + A_k}{\sum_{i=0}^{n}( \alpha\deg^+(v_i) + A_i)},
\end{equation}
where $\deg^+(v_i)$ is the number of children of $v_i$, and an edge is drawn from $v_k$ to a new child vertex $v_{n+1}$ to form $T_{n+1}$. Since $T_{n}$ has $n+1$ vertices and $n$ edges, the denominator of \eqref{eq:prefatt} simplifies to $\alpha n + \sum_{i=0}^{n}A_i.$

If $A_n=1$ for all $n$, then $T_0, T_1, T_2, \ldots$ is a sequence of linear preferential attachment trees; this class of random tree models includes random recursive trees ($\alpha = 0$) and random plane-oriented recursive trees ($\alpha =1$) \cite{MAHM:92, SZYM:87}. If the $A_i$ are random and $\alpha =0$, then a sequence of random weighted recursive trees is constructed as in \cite{LOOR:20A, MABR:19, SENI:21}. Allowing $\alpha > 0$ results in a sequence of preferential attachment trees with additive i.i.d. random fitness (this model appears for example in \cite{LOOR:20}).

We can further generalize these models of random trees by giving random lengths to the edges, similar to models from \cite{BOVA:06, BRDE:06}. For a random variable $L\geq 0$ such that $\mathbb{P}(L=0) < 1$, when attaching the vertex $v_n$ to the tree, we can give a `length' $L_n \sim L$ to the edge joining $v_n$ to its parent. The insertion depth $\Delta_n$ is then the sum of the lengths $L_k$ along the path from the root to $v_n$. The random variable $L_n$ can be independent from $A_n$, or we may have a joint distribution on $(A_n, L_n)$. The tree models described above correspond to the case $L=1$. 

We may construct these trees as random recursive metric spaces. For $n \geq 1$, consider blocks $\bm{B}_n$ where $B_n$ is the graph $K_2$ consisting of two vertices joined by an edge. One of the vertices is labelled $H'_n$ and we will call the other vertex $v_n$. The metric $D'_n$ on the vertices of $B_n$ is defined so that $D'_n(H_n', v_n) = L_n$. The probability measure $P'_n$ is defined such that $P'_n(H'_n) = \alpha/(\alpha + A_n)$ and $P'_n(v_n) = A_n/(\alpha + A_n)$. The weight of the block $\bm{B}_n$ is $W_n = \alpha + A_n$. Define $\bm{B}_0$ with $B_0$ consisting of a single vertex $v_0$ with probability measure $P'_0(v_0) =1$. Set the weight of $\bm{B}_0$ to be $W_0 = A_0$. This construction allows us to prove the following theorem. 

\begin{theorem}\label{thm:trees}
For $\alpha \geq 0$, let $T_0, T_1, T_2, \ldots$ be preferential attachment trees with fitness distribution $A$ and random edge lengths with distribution $L$ as described above, and let $\Delta_n$ be the insertion depth of the vertex $v_n$. If $\mathbb{E}\left[A^2\right], \mathbb{E}\left[A^2L^2\right]$ and $\mathbb{E}\left[AL^2\right]$ are all finite, then 
\[  \frac{\Delta_n}{\ln n} \xrightarrow{p} \frac{\mathbb{E}[AL]}{\alpha + \mathbb{E}[A]} \qquad \mathrm{and} \qquad \frac{\Delta_n - \left(\alpha + \mathbb{E}[A]\right)^{-1} \mathbb{E}[AL]\ln n}{\sqrt{\ln n}} \xrightarrow{d} \mathcal{N}\left(0, \frac{\mathbb{E}\left[AL^2\right]}{\alpha + \mathbb{E}[A]}\right).\]
\end{theorem}

\begin{proof}
We start by showing that the random recursive metric spaces $\bm{G}_0, \bm{G}_1, \bm{G}_2, \ldots$ constructed from the blocks $\bm{B}_n$ above are distributed as the trees $T_0, T_1, T_2, \ldots$. First we note that $G_n$ is a tree; every time a latch is attached with a hook, these two vertices are joined together, leaving a new child adjacent to the latch. To see that the trees constructed coincide with the random tree models described above, look at the vertex $v_k$ (belonging to the block $\bm{B}_k$), and suppose $I_{n,k} \subseteq \{k+1, \ldots, n\}$ is the set of indices of blocks $\bm{B}_i$ whose hook $H'_i$ is joined with $v_k$. Then the probability that $v_k$ is chosen as the latch to construct $\bm{G}_{n+1}$, by the definition of $P_n$ as the weighted sum of the probability measures $P'_0, P'_1, P'_2, \ldots$,  is given by 
\begin{align*}
P_n(v_k) &= \frac{W_k}{S_n}P'_k(v_k) + \sum_{i\in I_{n,k}}\frac{W_i}{S_n}P'_i(H_i') \\
&= \frac{1}{S_n}\left((\alpha + A_k)\frac{A_k}{\alpha + A_k} + \sum_{i \in I_{n,k}} (\alpha + A_i)\frac{\alpha}{\alpha + A_i}\right) \\
&= \frac{\alpha |I_{n,k}| + A_k}{S_n}.
\end{align*}
Since $|I_{n,k}| = \deg^+(v_k)$ and 
\[S_n = A_0 + \sum_{k=1}^{n}(\alpha + A_k) = \alpha n + \sum_{k=0}^n A_k,\]
we see that $P_{n}(v_k)$ is the same as in \eqref{eq:prefatt}, and so the two models are equal in distribution. 

Next we evaluate $W$ and $\Delta'$. We see that $\mathbb{E}[W] = \alpha + \mathbb{E}[A]$ and $\mathbb{E}\left[W^2\right] = \alpha^2 + 2\alpha\mathbb{E}[A] + \mathbb{E}\left[A^2\right]$. Conditioned on $(A_n, L_n)$, the random variable $\Delta'_n$ has distribution 
\[ \Delta'_n = \begin{cases}
0 & \text{ with probability } \frac{\alpha}{\alpha + A_n} \\
L_n & \text{ with probability } \frac{A_n}{\alpha + A_n}
\end{cases}\]
 and $\Delta'_0 = 0$. Quick calculations yield $\mathbb{E}[W\Delta'] = \mathbb{E}[AL]$ and $\mathbb{E}[W(\Delta')^2] = \mathbb{E}\left[AL^2\right]$ as well as $\mathbb{E}\left[(W\Delta')^2\right] = \alpha \mathbb{E}\left[AL^2\right] + \alpha\mathbb{E}\left[A^2L^2\right]$. The moment conditions to apply Theorem \ref{thm:main} are satisfied if $\mathbb{E}\left[A^2\right], \mathbb{E}\left[A^2L^2\right]$ and $\mathbb{E}\left[AL^2\right]$ are all finite, proving the theorem.
\end{proof}

When $L=1$ these results coincide with central limit theorems proved for random recursive trees when $\alpha = 0$ \cite{DEVR:88, MAHM:91}, for random plane-oriented recursive trees when $\alpha = 1$ \cite{MAHM:92}, and for random weighted recursive trees (where stronger results on the profile of vertices was proved in \cite{MABR:19, SENI:21}). A weak law of large numbers was proved for specific cases of random recursive trees with random edge lengths \cite{BOVA:06, HAAS:17}, but the general result of Theorem \ref{thm:trees} is new.

%---------- Hooking networks -----------%
\subsection{Hooking networks}

For a graph $G$, let $E(G)$ and $V(G)$ be the edge set and vertex set of $G$ respectively, and let $\deg(v)$ be the degree of a vertex in $G$. Let $\mathcal{C} = \{G_1, G_2, G_3, \ldots\}$ be a collection of finite connected graphs (self-loops and multi-edges are allowed) called blocks, each with an identified vertex $h_i$ called a hook, and each associated with a positive probability $p_i$ such that $\sum p_i = 1$. Let $\chi \geq 0$ and $\rho \in \mathbb{R}$ be two parameters with the condition that $\chi+ \rho > 0$. The hooking networks $\mathcal{H}_0, \mathcal{H}_1, \mathcal{H}_2, \ldots$ are constructed by first setting $\mathcal{H}_0$ as an isomorphic copy of a graph selected from $\mathcal{C}$ (at random or not). For $n\geq 0$, $\mathcal{H}_{n+1}$ is constructed from $\mathcal{H}_{n}$ by first selecting a vertex $v$ called a latch from $\mathcal{H}_{n}$ with probability 
\begin{equation}\label{eq:oldhook}
 \frac{\chi \deg(v) + \rho}{\sum_{u \in V(\mathcal{H}_{n})}( \chi \deg(u) + \rho)}.
\end{equation}
Once a latch is selected, a graph $G_i$ is selected from $\mathcal{C}$ with probability $p_i$, and an isomorphic copy of $G_i$ is attached by fusing the hook $h_i$ with the latch $v$. The selection probability in \eqref{eq:oldhook} is very similar to the selection probability \eqref{eq:prefatt} for trees. The difference arises from the fact that \eqref{eq:oldhook} is a function of the vertex degree instead of the number of children, and the condition $\chi + \rho > 0$ guarantees that this probability is positive and well defined. Hooking networks were introduced in \cite{MAHM:19}, where the set of blocks consisted of a single graph. The model was generalized in \cite {DEHO:20} to allow for finitely many blocks in $\mathcal{C}$. 

We may generalize hooking networks further by replacing $\rho$ with random fitnesses; whenever a graph $G_i$ is attached to the hooking network, every vertex $v \neq h_i$ from $G_i$ is assigned a random fitness $\rho_v$, such that $\chi\deg(v) + \rho_v \geq 0$ and $\mathbb{P}(\chi\deg(v) + \rho_v = 0) < 1$, independently from the construction of the hooking networks so far (though not necessarily independently from the other vertices within $G_i$). Then for future hooking, the probability of selecting $v \in \mathcal{H}_{n}$ to be a latch is given by 
\begin{equation}\label{eq:hook}
\frac{\chi \deg(v) + \rho_v}{\sum_{u \in V(\mathcal{H}_{n})}(\chi \deg(u) + \rho_u)}.
\end{equation}

We can construct random recursive metric spaces that are distributed as hooking networks. For $n\geq 0$, we define $\bm{B}_n$ so that $B_n$ is a graph $G_i$ selected at random from $\mathcal{C}$ with probability $p_i$, with an identified hook $H_n' = h_i$, and $D'_n$ is the standard graph metric on $B_n$. Let $\deg_{B_n}(v)$ be the degree of $v$ within $B_n$. Random fitnesses $\rho_v$ are then assigned to each vertex $v \in B_n$ such that when $v$ is not the hook, $\chi\deg_{B_n}(v) + \rho_v \geq 0$ and $\mathbb{P}(\chi\deg_{B_n}(v) + \rho_v = 0) < 1$, while for the hooks we have the condition $\chi\deg_{B_0}(H) + \rho_{H} > 0$ when $n=0$ and we set $\rho_{H'_n} = 0$ when $n\geq 1$ (since a hook will be fused with a vertex in the hooking network and only contribute $\chi\deg_{B_n}(H_n')$ to future hooking). The weight of $\bm{B}_n$ is given by 
\[ W_n = \sum_{v \in V(B_n)}(\chi \deg_{B_n}(v) + \rho_v) = 2\chi|E(B_n)| + \sum_{v \in V(B_n) }\rho_v,\]
where the last equality follows from the handshaking lemma.  The probability $P'_n$ is defined such that $P'_n(v) = (\chi \deg_{B_n}(v) + \rho_v)/W_n$. The condition $\chi\deg_{B_0}(H) + \rho_{H} > 0$ for $n=0$ guarantees that $W_0 > 0$.

Given the definition of the blocks $\bm{B}_n$, we can explicitly calculate the values $\mathbb{E}[W], \mathbb{E}[W\Delta']$, and $\mathbb{E}\left[W(\Delta')^2\right]$. To start we have 
\begin{equation}\label{eq:hookweight}
 \mathbb{E}[W] = \sum_{G_i \in \mathcal{C}} p_i \left(2\chi|E(G_i)| + \sum_{v \in V(G_i)}\mathbb{E}[\rho_v]\right).
\end{equation}
Conditioned on $\bm{B}_n$, the distribution of $\Delta_n'$ is given by
\[ \mathbb{P}\left(\Delta_n' = k | \bm{B}_n\right) =
\frac{1}{W_n}\sum_{v \in V(B_n)}(\chi\deg_{B_n}(v) + \rho_v)\bm{1}_{\{D_n'(H_n',v) = k\}},\]
and so 
\[ \condE{W_n\Delta'_n}{ \bm{B}_n} = W_n\condE{\Delta'_n }{ \bm{B}_n} = W_n \sum_{k=0}^\infty k \mathbb{P}\left(\Delta'_n = k | \bm{B}_n\right) = \sum_{v \in V(B_n)} D_n'(H_n',v)(\chi\deg_{B_n}(v) + \rho_v).\]
If we let $\delta_{G_i}(v)$ be the distance from $h_i$ to $v$ in the graph $G_i$, then by the law of total expectation, 
\begin{equation}\label{eq:hookexpdelta} \mathbb{E}[W\Delta'] = \sum_{G_i \in \mathcal{C}} p_i \sum_{v \in V(G_i)} \delta_{G_i}(v)\left(\chi\deg_{G_i}(v) + \mathbb{E}[\rho_v]\right),
\end{equation}
and by a similar argument
\begin{equation}\label{eq:hookvardelta}
\mathbb{E}\left[W(\Delta')^2\right] = \sum_{G_i \in \mathcal{C}} p_i \sum_{v \in V(G_i)} (\delta_{G_i}(v))^2\left(\chi\deg_{G_i}(v) + \mathbb{E}[\rho_v]\right).
\end{equation}
We can also develop equations for $\mathbb{E}\left[W^2\right]$ and $\mathbb{E}\left[(W\Delta')^2\right]$, but we simply note that these expressions are only dependent on $\mathcal{C}$, the probabilities $p_1, p_2, \ldots$, and the distributions of the $\rho_v$'s. We know what conditions are needed to apply Theorem \ref{thm:main} to the random recursive metric spaces $\bm{G}_0, \bm{G}_1, \bm{G}_2, \ldots$ constructed form the blocks $\bm{B}_n$. We simply need to show that these random recursive metric spaces are distributed as hooking networks. To avoid confusion with the blocks $G_i$, we will write $\bm{G}_n = (\mathcal{G}_n, D_n, H, P_n, S_n)$ for the random recursive metric spaces. 

\begin{theorem}\label{thm:hooking}
Let $\mathcal{C} = \{G_1, G_2, G_3, \ldots\}$ be a set of finite graphs with probabilities $p_1, p_2, p_3, \ldots$ that sum to 1, and independently for each graph $G_i$, let $\rho_v$ be a fitness distribution for each $v \in G_i$ satisfying the conditions above. Let $\mathcal{H}_0, \mathcal{H}_1, \mathcal{H}_2, \ldots$ be hooking networks constructed as above and let $\Delta_n$ be the insertion depth of the $n$'th block. If $\mathbb{E}\left[W^2\right], \mathbb{E}\left[(W\Delta')^2\right]$, and $\mathbb{E}\left[W(\Delta')^2\right]$ are all finite, then
 \[ \frac{\Delta_n}{\ln n} \xrightarrow{p} \frac{\mathbb{E}\left[W\Delta'\right]}{\mathbb{E}[W]} \qquad \mathrm{and} \qquad\frac{\Delta_n - \mathbb{E}[W]^{-1}\mathbb{E}\left[W\Delta'\right]\ln n}{\sqrt{\ln n}} \xrightarrow{d} \mathcal{N}\left(0,\frac{ \mathbb{E}\left[W(\Delta')^2\right]}{\mathbb{E}[W]}\right),\]
where $\mathbb{E}[W], \mathbb{E}\left[W\Delta'\right]$, and $\mathbb{E}\left[W(\Delta')^2\right]$ are given in \eqref{eq:hookweight}, \eqref{eq:hookexpdelta}, and \eqref{eq:hookvardelta} respectively. 
\end{theorem}

\begin{proof}
Similar to the proof of Theorem \ref{thm:trees}, we only need to show that the selection of a latch $v_n$ in the random recursive metric space $\bm{G}_n$ is made with probability \eqref{eq:hook} to prove they have the same distribution. Let $v$ be a vertex in $\mathcal{G}_n$, and let $k$ be the smallest value for which $v$ belongs to $B_k$ (recall that all hooks other than $H$ are fused with some other vertex in the hooking network). Let $I_{n,v} \subseteq \{k+1, \ldots, n\}$ be the set of indices of blocks $\bm{B}_i$ whose hook $H_i'$ is fused with $v$. Since $\rho_{H_i'} = 0$ for all $i \in I_{n,v}$, the probability of selecting $v$ as a latch to construct $\mathcal{G}_{n+1}$  is given by 
\[ P_n(v) = \frac{W_k}{S_n}P'_k(v) + \sum_{i\in I_{n,v}} \frac{W_i}{S_n}P_i'(H_i') = \frac{1}{S_n}\left(\chi \deg_{B_k}(v) + \rho_v + \sum_{i \in I_{n,v}} \chi\deg_{B_i}(H_i')\right). \]
Since fusing a hook $H'_i$ to $v$ increases the degree of $v$ by $\deg_i(H_i')$ in the hooking network, then $\chi \deg_{B_k}(v) + \rho_v + \sum_{i \in I_{n,v}} \chi\deg_{B_i}(H_i') = \chi\deg(v) + \rho_v$, where $\deg(v)$ is the degree of $v$ in $\mathcal{G}_n$, which is the numerator in \eqref{eq:hook}. Every time a block $\bm{B}_k$ is attached to the hooking network, $|E(B_k)|$ new edges are added, and using again that $\rho_{H'_n} = 0$ for all $n \geq 1$,
\[ S_n = \sum_{k=0}^n W_k= \sum_{k=0}^n\left(2\chi|E(B_n)| + \sum_{v \in V(B_k)}\rho_v\right) = 2\chi|E(\mathcal{G}_n)| + \sum_{v \in V(\mathcal{G}_n)}\rho_v.\]
By applying the handshaking lemma, $S_n$ is the denominator in \eqref{eq:hook}. Thus the probability of choosing any vertex $v \in V(\mathcal{G}_n)$ to be a latch is equal to \eqref{eq:hook} and the random recursive metric spaces are distributed as hooking networks. Applying Theorem \ref{thm:main} completes the proof.
\end{proof}

The classic hooking networks are recovered if the $\rho_v$ are all equal to a deterministic $\rho$ for all vertices $v$ other than $H_n'$ for $n \geq 1$ (and $\rho_{H_n'} =0$). If $\mathcal{C}$ is also a finite collection of finite graphs, then it is evident that the moment conditions of Theorem \ref{thm:hooking} are satisfied. When the weights $W_n$ are deterministic, which holds, for example, when $\chi =0$ and all blocks have the same number of vertices, or when $\rho =0$ for all vertices and all blocks have the same number of edges, a normal limit law for the insertion depth was proved in \cite{DEMA:21} (where the term {\em affinity} is used for the weight of the blocks). Theorem \ref{thm:main} generalizes this result in the classic case by allowing for blocks that produce random weights $W_n$. When $\rho_v$ are random, the moment conditions for Theorem \ref{thm:hooking} are also met if $\mathcal{C}$ is finite and if $\mathbb{E}\left[\rho^2_v\right] < \infty$ and $\mathbb{E}\left[\rho_v\rho_u\right] < \infty$ for all $u,v$ in a graph $G_i$.  

The set of blocks $\mathcal{C}$ can also be infinite. We present an example where the blocks are paths of geometric lengths. Fix $p \in (0,1)$ and let $\chi =0$ and $\rho =1$ (so latches are chosen uniformly at random among all the vertices when growing the hooking network). Let $W_n \sim Ge(p)$, and let $B_n$ be a path of length $W_n$ (with $W_n+1$ vertices). Then the weight of this block is $W_n$ and $P'_n(v) = 1/W_n$ for $v \neq H_n'$ (and $P'_n(H'_n) = 0$). Given $W_n$, the random variable $\Delta'_n$ is distributed as $\mathbb{P}(\Delta'_n = j) = 1/W_n$ for $j=1, 2, \ldots, n$. Quick calculations reveal $\mathbb{E}[W] = 1/p,$ $\mathbb{E}[W\Delta'] = 1/p^2$, and $\mathbb{E}[W(\Delta')^2] = (2-p)/p^3$. Applying Theorem \ref{thm:main} gives 
\[ \frac{\Delta_n}{\ln n} \xrightarrow{p} \frac{1}{p}\qquad \mathrm{and} \qquad \frac{\Delta_n - \frac{1}{p}\ln n}{\sqrt{\ln n}} \xrightarrow{d} \mathcal{N}\left(0, \frac{2-p}{p^2}\right).\]

We note that we may apply Theorem \ref{thm:main} to further generalizations of hooking networks; for example, we may also assign random lengths to the edges (similar to the random lengths added to edges of trees in Section \ref{sec:trees}), though we note that the shortest path from one vertex to another may change depending on the edge lengths.

%----------- Line Segments ----------%
\subsection{Constructions from line segments}

In the applications we have seen so far, the blocks have been discrete metric spaces. But we can certainly take continuous spaces as well. As a simple example, we present random trees constructed from line segments similar to those in \cite{ADGO:17,CUHA:17, HAAS:17}. 

Let $W, W_1, W_2, \ldots$ be i.i.d. non-negative random variables with $\mathbb{P}[W = 0] < 1$, and let $B_n$ be a line segment of length $W_n$. We recursively construct the spaces $\mathcal{T}_0, \mathcal{T}_1, \mathcal{T}_2, \ldots$ in the following way: $\mathcal{T}_0$ is the line segment $B_0$ with one endpoint called a root, and for $n\geq 0$, sample a point $v$ uniformly at random from $\mathcal{T}_{n}$ and glue the segment $B_n$ to $v$. 
A law of large numbers for the insertion depth was proved for the case $W=1$ in \cite{HAAS:17}. We apply Theorem \ref{thm:main} to the more general spaces $\mathcal{T}_n$.  

In the notation of random recursive metric spaces, define the blocks $\bm{B}_n$ in the following way. Let $B_n$ a copy of the line segment $[0,W_n]$ with euclidean metric $D_n'$ and let the point 0 be the hook $H_n'$. Let $P_n'$ be uniform on $B_n$ (so $P_n'$ has the same distribution as the uniform distribution $U[0,W_n]$). Set the weight of $\bm{B}_n$ as $W_n$. We can show that the random recursive metric spaces are distributed as $\mathcal{T}_n$, and so we just need to find the moment conditions on $W$ needed to apply Theorem \ref{thm:main}. 

\begin{theorem}\label{thm:segments}
Let $\mathcal{T}_0, \mathcal{T}_1, \mathcal{T}_2, \ldots$ be the random spaces defined above and let $\Delta_n$ be the distance from the root to where the $n$'th segment is glued. If $\mathbb{E}\left[W^4\right] < \infty$, then 
\[ \frac{\Delta_n}{\ln n} \xrightarrow{p} \frac{\mathbb{E}[W^2]}{2\mathbb{E}[W]} \qquad \mathrm{and} \qquad \frac{\Delta_n - \frac{1}{2}\left(\mathbb{E}[W]\right)^{-1}\mathbb{E}[W^2]\ln n}{\sqrt{\ln n}} \xrightarrow{d} \mathcal{N}\left( 0, \frac{\mathbb{E}[W^3]}{3\mathbb{E}[W]}\right). \]
\end{theorem}

\begin{proof}
It is evident that the random recursive metric spaces constructed from $\bm{B}_n$ are distributed as $\mathcal{T}_0, \mathcal{T}_1, \mathcal{T}_2, \ldots$; at each step $n \geq 0$ we can sample $v$ from $\mathcal{T}_{n}$ by first sampling the segment $B_k$ with probability $W_k/S_{n}$ and sampling a point uniformly in $B_k$. From the first two moments of the uniform distribution we get that 
\[ \condE{W_n\Delta'_n}{W_n} = \frac{W_n^2}{2}, \qquad \condE{W_n(\Delta'_n)^2}{W_n} = \frac{W_n^3}{3}, \qquad \condE{(W_n\Delta'_n)^2}{W_n} = \frac{W_n^4}{3},\]
and so from the law of total expectation
\[ \mathbb{E}\left[W\Delta' \right]  = \frac{\mathbb{E}\left[ W^2\right]}{2}, \qquad  \mathbb{E}\left[W(\Delta')^2 \right] = \frac{\mathbb{E}\left[ W^3\right]}{3},  \qquad  \mathbb{E}\left[(W\Delta')^2 \right]= \frac{\mathbb{E}\left[ W^4\right]}{3}\]
Thus we apply Theorem \ref{thm:main} if $\mathbb{E}\left[W^4\right] < \infty$, proving the theorem.
\end{proof}

%----------------------------------%
%---------- PROOFS ----------%
%----------------------------------%
\section{Proof of Theorem \ref{thm:main}}\label{sec:proofs}

Given the blocks $\bm{B}_0, \bm{B}_1, \bm{B}_2, \ldots$, for every $n$ define for $k=0, 1, \ldots, n-1$ the sequence of independent (conditioned on the blocks) random variables $Y_{n,k} = J_{n,k}\Delta'_{n,k}$, where 
\[J_{n,k} \sim \text{Be}\left(\frac{W_k}{S_k}\right),\]
$\Delta'_{n,k} \sim \Delta_k'$ for $k\geq 1$, and $\Delta'_{n,0} \sim \Delta'_0$. Notice that $J_{n,0} = 1$. Conditioned on $\bm{B}_0, \bm{B}_1, \bm{B}_2, \ldots$, the random variables $J_{n,k}$ and $\Delta'_{n,k}$ are taken to be independent (though of course $W_k$ and $\Delta'_{n,k} \sim \Delta'_k$ are dependent on the block $\bm{B}_k$). 

As mentioned in the introduction, several previous works have described the insertion depth in related models as sums of Bernoulli random variables. Here we provide a similar observation for random recursive metric spaces, in this case describing the insertion depth as the sum of $Y_{n,k}$. Let $\bm{G}_0, \bm{G}_1, \bm{G}_2, \ldots$ be random recursive metric spaces constructed from the blocks $\bm{B}_n$. By the definition of the probability measure $P_n$, we can choose the latch $v_{n+1}$ by first choosing a block $\bm{B}_k$ already attached to $G_n$ with probability $W_k/S_n$, and then choosing the latch $v_{n+1}$ within $B_k$ according to $P'_k$. Once the new block $\bm{B}_{n+1}$ is attached to $v_{n+1}$, we will say that the block $\bm{B}_{n+1}$ is a child of the block $\bm{B}_k$. Similarly define the notion of blocks as parents, descendants, and ancestors. For a fixed $n$, let $0 = k_0< k_1 < \cdots < k_t$ be all the values $k_i$ for which $J_{n,k_i} = 1$, and let $k_{t+1} = n$. We will develop a coupling between the collection of random variables $Y_{n,k}, J_{n,k},\Delta'_{n,k}$, and random recursive metric spaces, whereby $J_{n,k} = 1$ if and only if $\bm{B}_k$ is an ancestor of $\bm{B}_n$ in the random recursive metric space and $D_n(v_{k_i}, v_{k_{i+1}})$ is distributed as $\Delta_{n,k_{i}}'$ for all $i=0,\ldots, t$. If this holds, then 
\[ \Delta_n = \sum_{i=0}^{t} D_n(v_{k_i}, v_{k_{i+1}}) \sim \sum_{k=0}^{n-1} J_{n,k}\Delta'_{n,k} =\sum_{k=0}^{n-1} Y_{n,k}.\]
The following lemma is contained in the proof of \cite[Lemma 3]{SENI:19}, but we provide a proof for the sake of completeness.

\begin{lemma}\label{lem:buckets}
With $Y_{n,k}$ defined above, the insertion depth $\Delta_n$ of the $n'th$ block added to a random recursive metric space $\bm{G}_n$ is distributed as 
\[\Delta_n \sim\sum_{k=0}^{n-1} Y_{n,k}.\]
\end{lemma}

\begin{proof}
Fix the step $n$ throughout this proof. For a sequence of blocks $\bm{B}_0, \bm{B}_1, \ldots, \bm{B}_n$, sample the random variables $J_{n,k}$, $\Delta'_{n,k}$ and $Y_{n,k}$ as described above. We construct a sequence of random recursive metric spaces $\bm{G}_{0}, \bm{G}_{1},  \ldots, \bm{G}_{n}$. Set $\bm{G}_{0} = \bm{B}_0$. For $k= 1, 2, \ldots, n$, choose the latch $v_{k}$ by first choosing a block $\bm{B}_{\ell_k}$ in the following way:
\begin{itemize}
\item If $k=n$ or $J_{n,k} =1$, then $\ell_k = \max_{0 \leq k' < k}\{J_{n,k'} = 1\}$. This is always defined since $J_{n,0} = 1$. Next choose the latch $v_k$ within $B_{\ell_k}$ according to $P_{\ell_k}'$ conditioned on $D_{\ell_k}'(H'_{\ell_k}, v_k) = \Delta'_{n,\ell_k}$ (this is well defined by the Radon-Nikodym theorem since $\{x \in B_{\ell_k} : D'_{\ell_k}(H_{\ell_k}, x) = \Delta'_{n,\ell_k}\}$ is a closed set in $B_{\ell_k}$, and so belongs to the sigma algebra for which $P_{\ell_k}'$ is defined). 
\item If $k \neq n$ and $J_{n,k} = 0$, then $\ell_k$ is chosen at random among $0, 1, 2, \ldots, k-1$ with probability $W_{\ell_k}/S_{k-1}$, and $v_k$ is chosen in $B_{\ell_k}$ according to $P_{\ell_k}'$. 
\end{itemize}

If $0 = k_0< k_1 < \cdots < k_t$ are all the values $k_i$ for which $J_{n,k_i} = 1$, notice that the parent of $\bm{B}_n$ is $\bm{B}_{k_t}$, the parent of $\bm{B}_{k_t}$ is $\bm{B}_{k_{t-1}}$, and so on until $\bm{B}_0$. Conversely, if $J_{n,k} = 0$, then $\bm{B}_k$ cannot be the parent of any of $\bm{B}_{k_1}, \bm{B}_{k_2}, \ldots, \bm{B}_{k_t}, \bm{B}_n$. Therefore, the block $\bm{B}_k$ is an ancestor of $\bm{B}_n$ if and only if $J_{n,k} = 1$. 

To prove that $\bm{G}_{0}, \bm{G}_{n}, \ldots, \bm{G}_{n}$ is distributed as random recursive metric spaces, we need only show that the choice of $\ell_k$ is made with probability $W_{\ell_k}/S_{k-1}$. Indeed we see that for $k < n$, $\ell_k$ is chosen with probability 
\begin{equation}\label{eq:l}
 \mathbb{P}\left(J_{n,k} = 1, J_{n,\ell_k} = 1, J_{n,\ell_k+1}=\cdots=J_{n,k-1} =0\right) + \mathbb{P}(J_{n,k} = 0)\frac{W_{\ell_k}}{S_{k-1}}, 
\end{equation}
and when $k=n$, $\ell_k < n-1$ is chosen with probability 
\begin{equation}\label{eq:ln}
 \mathbb{P}\left(J_{n,\ell_k} = 1, J_{n,\ell_k+1}=\cdots=J_{n,k-1} =0\right),
\end{equation}
and $\ell_k = n-1$ is chosen with probability $\mathbb{P}(J_{n,n-1} = 1) = W_{n-1}/S_{n-1} = W_{\ell_k}/S_{k-1}$. 
Since 
\begin{equation*}
\mathbb{P}\left(J_{n,\ell_k} = 1, J_{n,\ell_k+1}=\cdots=J_{n,k-1} =0\right) =  \frac{W_{\ell_k}}{S_{\ell_k}} \prod_{k'=\ell_k+1}^{k-1}\left(1 - \frac{W_{k'}}{S_{k'}}\right) = \frac{W_{\ell_k}}{S_\ell}\prod_{k'={\ell_k}+1}^{k-1}\frac{S_{k'-1}}{S_{k'}} = \frac{W_{\ell_k}}{S_{k-1}},
\end{equation*}
both Equations \eqref{eq:l} and \eqref{eq:ln} simplify to $W_{\ell_k}/S_{k-1}$. Therefore, we have indeed constructed a sequence of random recursive metric spaces and from the construction described, we see that 
\[\Delta_n = D_n(H, v_n) \sim \sum_{k=0}^{n-1}Y_{n,k}. \qedhere\]
\end{proof}

We now introduce the martingales we will use to prove Theorem \ref{thm:main}. For every $n\geq 1$ and $k=0, \ldots, n-1$, let $\mathcal{F}_{n,k}$ be the $\sigma$-algebra generated by 
$W_1, \ldots, W_k$ and all $Y_{m,\ell}$ defined for $m \leq n, \ell \leq k$. Note that the filtrations $(\mathcal{F}_{n,k})_{k=0}^{n-1}$ are nested (that is, $\mathcal{F}_{n,k} \subseteq \mathcal{F}_{n+1, k}$). Define
\[ M_{n,k} = \sum_{\ell=1}^k \left(Y_{n,\ell} - \condE{Y_{n,\ell}}{\mathcal{F}_{n,\ell-1}}\right) \qquad\mathrm{and}\qquad L_{n,k}= \sum_{\ell=1}^k \condE{Y_{n,\ell}}{\mathcal{F}_{n,\ell-1}},\]
so that $\sum_{\ell=0}^k Y_{n,\ell} = Y_{n,0} + M_{n,k} + L_{n,k}$. Then by Lemma \ref{lem:buckets}, $\Delta_n \sim Y_{n,0} + M_{n,n-1} + L_{n,n-1}$. We will apply a martingale central limit theorem, namely \cite[Corollary 3.1]{HAHE:80}, to $\{M_{n,k}, 1 \leq k < n, n \geq 2\}$ scaled by $\sqrt{\ln n}$ to prove Theorem \ref{thm:main}. We begin with some useful calculations.

\begin{lemma}
Suppose that $\mathbb{E}\left[W^2\right], \mathbb{E}\left[(W\Delta')^2\right]$, and $\mathbb{E}\left[W(\Delta')^2\right]$ are all finite. Then for almost every sequence $\bm{B}_0, \bm{B}_1, \bm{B}_2, \ldots$ of blocks, the following hold.
\begin{equation}\label{eq:expconY}
\condE{Y_{n,k}}{\mathcal{F}_{n,k-1}}=\frac{\mathbb{E}[W\Delta']}{k \mathbb{E}[W]} +O(k^{-2}),
\end{equation}
\begin{equation}\label{eq:varconY}
\condE{\left(Y_{n,k} - \mathbb{E}[Y_{n,k} | \mathcal{F}_{n,k-1}]\right)^2}{\mathcal{F}_{n,k-1}} =\frac{\mathbb{E}[W(\Delta')^2]}{k \mathbb{E}[W]} +O(k^{-2}),
\end{equation}
and for $c > 0$,
\begin{equation}\label{eq:lindY}
\condE{Y_{n,k}^2\bm{1}_{\{ Y_{n,k} > c\}}}{\mathcal{F}_{n,k-1}}= \frac{\mathbb{E}[W(\Delta')^2\bm{1}_{\{\Delta' > c\}}]}{k \mathbb{E}[W]} +O(k^{-2}).
\end{equation}
\end{lemma}

\begin{proof}
Throughout this argument we use that for almost every sequence $W_1, W_2, \ldots$ of weights, we have $S_k = k\mathbb{E}[W] + o(k^{3/4})$. This is guaranteed for example by the law of the iterated algorithm, since the $W_i$ are i.i.d. and have finite variances. Also note that given our conditions $\mathbb{E}\left[W^2\right] < \infty$ and $\mathbb{E}\left[(W\Delta')^2\right] < \infty$, since $W$ and $\Delta'$ are nonnegative, we also have 
\[ \mathbb{E}\left[W^2\Delta'\right] = \mathbb{E}\left[W^2\Delta'\bm{1}_{\{\Delta' \leq 1\}}\right] + \mathbb{E}\left[W^2\Delta'\bm{1}_{\{\Delta' > 1\}}\right] \leq \mathbb{E}\left[W^2\right] + \mathbb{E}\left[(W\Delta')^2\right] < \infty.\]
Define $\mathcal{F}_{n,k-1}^\ast = \sigma\left(\mathcal{F}_{n,k-1} \cup \sigma(W_k, \Delta_{n,k}')\right).$ Then 
\[ \condE{Y_{n,k}}{\mathcal{F}_{n,k-1}^\ast} =  \frac{W_k}{S_{k-1} + W_k}\Delta_{n,k}' = \Delta_{n,k}'\left( \frac{W_k}{S_{k-1}} + \frac{W_k^2}{S_{k-1}^2 + S_{k-1}W_k}\right).\]
Since $S_{k-1}$ and $W_k$ are nonnegative, the second term in the parentheses above is bounded by $W_k^2S_{k-1}^{-2}$. From our approximation for $S_k$ and since $\mathbb{E}\left[W^2 \Delta'\right]< \infty$, applying the tower rule yields that almost surely 
\[ \condE{Y_{n,k}}{\mathcal{F}_{n,k-1}} = \condE{\condE{Y_{n,k}}{\mathcal{F}_{n,k-1}^\ast}}{\mathcal{F}_{n,{k-1}}} = \frac{\mathbb{E}[W\Delta']}{k \mathbb{E}[W]} + O\left(k^{-2}\right),\]
which is \eqref{eq:expconY}. 
Similarly we have the conditional expectation
\[\condE{Y_{n,k}^2 }{ \mathcal{F}_{n,k-1}^\ast}=\frac{W_k}{S_{k-1} + W_k}(\Delta'_{n,k})^2,\]
and from our approximation for $S_k$ and the assumption $\mathbb{E}\left[(W\Delta')^2\right] < \infty$, applying the tower rule yields that almost surely
\[ \condE{Y_{n,k}^2}{\mathcal{F}_{n,{k-1}}} = \frac{\mathbb{E}[W (\Delta')^2]}{k \mathbb{E}[W]} + O\left(k^{-2}\right).\]
Subtracting $\left(\condE{Y_{n,k}}{\mathcal{F}_{n,{k-1}}}\right)^2 = O\left(k^{-2}\right)$ from the equation above yields \eqref{eq:varconY}. Finally, note that $Y_{n,k} > c$ if and only if $J_{n,k} = 1$ and $\Delta_{n,k}' > c$. Thus 
\[ \condE{Y_{n,k}^2\bm{1}_{\{ Y_{n,k} > c\}}}{ \mathcal{F}^\ast_{n,k-1}} = \frac{W_k}{S_{k-1} + W_k}\left(\Delta_{n,k}'\right)^2\bm{1}_{\{ \Delta'_{n,k} > c\}}.\]
Once more using the tower rule provides that \eqref{eq:lindY} holds almost surely.
\end{proof}

We are now ready to prove the main theorem of this work. 

\begin{proof}[Proof of Theorem \ref{thm:main}]
We start by proving that $\left\{(\ln n)^{-1/2}M_{n,k}, 1 \leq k < n, n \geq 2\right\}$ satisfies the conditions necessary to apply \cite[Corollary 3.1]{HAHE:80}. Note that 
\begin{equation}\label{eq:MsareYs}
M_{n,k} - M_{n,k-1} = Y_{n,k} - \condE{Y_{n,k}}{\mathcal{F}_{n,k-1}}.
\end{equation}
We see immediately that $\mathbb{E}[M_{n,k}] = 0$, and from \eqref{eq:MsareYs} we have that $\condE{M_{n,k} - M_{n,k-1}}{\mathcal{F}_{n,k-1}} = 0.$
From the moment conditions stated in Theorem \ref{thm:main}, we have also that $\mathbb{E}\left[M_{n,k}^2\right] < \infty$ for all $n$ and $k$. Thus $\left\{(\ln n)^{-1/2}M_{n,k}, 1 \leq k < n, n \geq 2\right\}$ is a zero-mean square-integrable martingale array. 

For the conditional variance, we conclude from \eqref{eq:varconY} and \eqref{eq:MsareYs} that almost surely, 
\begin{align*}
 \sum_{k=2}^{n-1} \condE{\left(M_{n,k} - M_{n,k-1}\right)^2}{\mathcal{F}_{n,k-1}} &=\sum_{k=2}^{n-1} \condE{\left(Y_{n,k} - \condE{Y_{n,k}}{\mathcal{F}_{n,k-1}}\right)^2}{\mathcal{F}_{n,k-1}}\\
& = \frac{\mathbb{E}\left[W(\Delta')^2\right]}{\mathbb{E}[W]} \ln n+ O(1). 
\end{align*}
Therefore,
\[ \sum_{k=2}^{n-1} \condE{\left(\frac{M_{n,k} - M_{n,k-1}}{\sqrt{\ln n}}\right)^2}{\mathcal{F}_{n,k-1}} \xrightarrow{p} \frac{\mathbb{E}\left[W(\Delta')^2\right]}{\mathbb{E}[W]}.\]

As for the conditional Lindeberg condition, we start by noting that for any $c>0$, \eqref{eq:expconY} guarantees that almost surely, for large enough $k$ we have $\condE{Y_{n,k}}{\mathcal{F}_{n,k-1}} < c$. Thus almost surely for large enough $k$, $\left\lvert Y_{n,k} - \condE{Y_{n,k}}{\mathcal{F}_{n,k-1}}\right\rvert > c$ implies that $Y_{n,k} > c $ and that $\left(Y_{n,k} - \condE{Y_{n,k}}{\mathcal{F}_{n,k-1}}\right)^2 < Y_{n,k}^2.$
Thus from \eqref{eq:lindY} and \eqref{eq:MsareYs}, we have that almost surely, 
\begin{align*}
\sum_{k=2}^{n-1}\condE{\left(M_{n,k} - M_{n,k-1}\right)^2\bm{1}_{\{|M_{n,k} - M_{n,k-1}| > c\}}}{\mathcal{F}_{n,k-1}} &< \sum_{k=2}^{n-1}\condE{Y_{n,k}^2\bm{1}_{\{Y_{n,k} > c\}}}{\mathcal{F}_{n,k-1}}\\
&= \frac{\mathbb{E}[W(\Delta')^2\bm{1}_{\{\Delta' > c\}}]}{\mathbb{E}[W]}\ln n +O(1).
\end{align*}
For any $\varepsilon > 0$, since $\mathbb{E}\left[W(\Delta')^2\right] < \infty$, we have that $\mathbb{E}\left[ W(\Delta')^2\bm{1}_{\{\Delta' > \varepsilon \sqrt{\ln n}\}}\right] \to 0.$ Thus
\[ \sum_{k=2}^{n-1}\condE{\left(\frac{M_{n,k} - M_{n,k-1}}{\sqrt{\ln n}}\right)^2\bm{1}_{\{|M_{n,k} - M_{n,k-1}| > \varepsilon \sqrt{\ln n}\}}}{\mathcal{F}_{n,k-1}} \xrightarrow{p} 0.\]
All the conditions of \cite[Corollary 3.1]{HAHE:80} are satisfied, and so 
\begin{equation}\label{eq:normalM}
 \frac{M_{n,n-1}}{\sqrt{\ln n}} \xrightarrow{d} \mathcal{N}\left( 0, \frac{\mathbb{E}\left[W(\Delta')^2\right]}{\mathbb{E}[W]}\right).
\end{equation}
By applying \eqref{eq:expconY}, we have that almost surely, 
\begin{equation}\label{eq:sumL}
 L_{n,n-1} = \sum_{k=1}^{n-1} \condE{Y_{n,k}}{\mathcal{F}_{n,k-1}} = \frac{\mathbb{E}\left[W\Delta'\right]}{\mathbb{E}[W]}\ln n + O(1),
\end{equation}
Since $\Delta_n \sim Y_{n,0} + M_{n,n-1} + L_{n,n-1}$, we can conclude from \eqref{eq:normalM} and \eqref{eq:sumL} the central limit theorem
\[ \frac{\Delta_n - \mathbb{E}[W]^{-1}\mathbb{E}\left[W\Delta'\right]}{\sqrt{\ln n}} \xrightarrow{d} \mathcal{N}\left( 0, \frac{\mathbb{E}\left[W(\Delta')^2\right]}{\mathbb{E}[W]}\right).\]
As for the weak law of large numbers, we have $M_{n,n-1}/\ln n \xrightarrow{d} 0$ from \eqref{eq:normalM} and Slutsky's theorem. Since we have convergence in distribution to a constant, the convergence holds in probability as well. From \eqref{eq:sumL}, we can conclude that $L_{n,n-1}/\ln n \xrightarrow{p} \mathbb{E}[W]^{-1}\mathbb{E}\left[W\Delta'\right],$ and so 
\[ \frac{\Delta_n}{\ln n} = \frac{Y_{n,0} + M_{n,n-1} + L_{n,n-1}}{\ln n} \xrightarrow{p} \frac{\mathbb{E}\left[W\Delta'\right]}{\mathbb{E}[W]}. \qedhere\]
\end{proof}

\section*{Acknowledgment}

The author would like to thank the referees for their comments in improving the presentation of this work, and in particular, an anonymous referee for suggesting the martingale approach for the proof of Theorem \ref{thm:main}. The author would also like to thank Svante Janson for valuable and helpful discussions in preparation of this work.

\end{document}